\documentclass[preprint,12pt]{elsarticle}

\usepackage{amssymb}
\usepackage{amsthm}

\newcommand{\sysn}{\left\{\begin{array}{rcl}}
\newcommand{\sysk}{\end{array}\right.}

\newtheorem{theorem}{Theorem}[section]

\theoremstyle{example}
\newtheorem{example}[theorem]{Example}

\theoremstyle{definition}

\newtheorem{corollary}[theorem]{Corollary}

\begin{document}

\begin{frontmatter}



\title{On separability of the functional space with the open-point and bi-point-open
topologies, II}


\author{Alexander V. Osipov}

\ead{OAB@list.ru}

\address{Ural Federal
 University, Institute of Mathematics and Mechanics, Ural Branch of the Russian Academy of Sciences, 16,
 S.Kovalevskaja street, 620219, Ekaterinburg, Russia}

\begin{abstract}
In this paper we continue to study the property of separability of
functional space $C(X)$ with the open-point and bi-point-open
topologies. We show that for every perfect Polish space $X$ a set
$C(X)$ with bi-point-open topology is a separable space. We also
show in a set model (the iterated perfect set model) that for
every regular space $X$ with countable network  a set $C(X)$ with
bi-point-open topology is a separable space if and only if a
dispersion character $\Delta(X)=\mathfrak{c}$.
\end{abstract}


\end{frontmatter}



\section{Introduction}

The space $C(X)$ with the topology of pointwise convergence is
denoted by $C_{p}(X)$. It has a subbase consisting of sets of the
form

$[x,U]^{+}=\{f\in C(X): f(x)\in U\}$,

where $x\in X$ and $U$ is an open subset of real line
$\mathbb{R}$.

In paper \cite{amk} was introduced two new topologies on $C(X)$
that we call the open-point topology and the bi-point-open
topology. The open-point topology on $C(X)$ has a subbase
consisting of sets of the form

$[V,r]^{-}=\{f\in C(X): f^{-1}(r)\bigcap V\neq\emptyset\}$,

where $V$ is an open subset of $X$ and $r\in \mathbb{R}$. The
open-point topology on $C(X)$ is denoted by $h$ and the space
$C(X)$ equipped with the open-point topology $h$ is denoted by
$C_{h}(X)$.

Now the bi-point-open topology on $C(X)$ is the join of the
point-open topology $p$ and the open-point topology $h$. It is the
topology having subbase open sets of both kind: $[x,U]^{+}$ and
$[V,r]^{-}$, where $x\in X$ and $U$ is an open subset of
$\mathbb{R}$, while $V$ is an open subset of $X$ and $r\in
\mathbb{R}$. The bi-point-open topology on the space $C(X)$ is
denoted by $ph$ and the space $C(X)$ equipped with the
bi-point-open topology $ph$ is denoted by $C_{ph}(X)$. One can
also view the bi-point-open topology on $C(X)$ as the weak
topology on $C(X)$ generated by the identity maps $id_{1}:
C(X)\mapsto C_{p}(X)$ and $id_{2}: C(X)\mapsto C_{h}(X)$.

 In \cite{amk} and \cite{amk1}, the separation and countability properties of these two
 topologies on $C(X)$ have been studied.

In \cite{amk} the following statements were proved.

\medskip

$\bullet$ $C_{h}(\mathbb{P})$ is separable where $\mathbb{P}$ is
the set of irrational numbers. (Proposition 5.1.)

$\bullet$ If $C_{h}(X)$ is separable, then every open subset of
$X$ is uncountable. (Theorem 5.2.)

$\bullet$ If $X$ has a countable $\pi$-base consisting of
nontrivial connected sets, then $C_{h}(X)$ is separable. (Theorem
5.5.)

$\bullet$ If $C_{ph}(X)$ is separable, then every open subset of
$X$ is uncountable. (Theorem 5.8.)

$\bullet$ If $X$ has a countable $\pi$-base consisting of
nontrivial connected sets and a coarser metrizable topology, then
$C_{ph}(X)$ is separable. (Theorem 5.10.)

\medskip

In \cite{os1} was definition a $\mathcal{I}$-set and was proved
that necessary condition for $C_{h}(X)$ be a separable space is
condition: $X$ has a $\pi$-network consisting of
$\mathcal{I}$-sets.

\medskip

$\bullet$  A set $A\subseteq X$ be called {\it $\mathcal{I}$-set}
if there is a continuous function $f\in C(X)$ such that $f(A)$
contains an interval $\mathcal{I}=[a,b]\subset \mathbb{R}$.

$\bullet$ If $C_{h}(X)$ is separable space, then $X$ has a
$\pi$-network consisting of $\mathcal{I}$-sets. (Theorem 2.3.)

\medskip

In this paper we use the following conventions. The symbols
$\mathbb{R}$, $\mathbb{P}$, $\mathbb{Q}$ and $\mathbb{N}$ denote
the space of real numbers, irrational numbers, rational numbers
and natural numbers, respectively. Recall that a dispersion
character $\Delta(X)$ of $X$ is the minimum of cardinalities of
its nonempty open subsets.

Recall also that a space be called  Polish space if it is a
separable complete metrizable space.

 By {\it set of reals} we mean a zero-dimensional,
separable metrizable space every non-empty open set which has the
cardinality the continuum.

\section{Main results}

Note that if the space $C_{h}(X)$ is a separable space then
$\Delta(X)\geq\mathfrak{c}$. Really, if $A=\{f_{i}\}$ is a
countable dense set of $C_{h}(X)$ then for each non-empty open set
$U$ of $X$ we have $\bigcup f_{i}(U)=\mathbb{R}$. It follows that
$|U|\geq\mathfrak{c}$.

Also note that if the space $C_{ph}(X)$ is a separable space then
$C_{p}(X)$ is a separable space and $C_{h}(X)$ is a separable. It
follows that $X$ is a separable submetrizable (coarser separable
metric topology) space and $\Delta(X)=\mathfrak{c}$.

Recall that a family $\gamma$ of subsets of a space $X$ is called
$T_0$-separating if whenever $x$ and $y$ are distinct points of
$X$, there exists $V\in \gamma$ containing exactly one of the
points $x$ and $y$.

In \cite{os1} was proved the following result.

\begin{theorem}  If $X$ is a Tychonoff space with network consisting
non-trivial connected sets, then the following are equivalent.

\begin{enumerate}

\item $C_{ph}(X)$ is a separable space.

\item $X$ is a separable submetrizable space.

\end{enumerate}
\end{theorem}

 In the present paper, we consider more wider form this theorem.

\begin{theorem}\label{th2} If $X$ is a Tychonoff space with $\pi$-network consisting
non-trivial connected sets, then the following are equivalent.
\begin{enumerate}

\item $C_{ph}(X)$ is a separable space.

\item $X$ has a countable $T_0$-separating family of zero-sets.

\item $X$ is a separable submetrizable space.

\end{enumerate}

\end{theorem}

\begin{proof}

$(1)\Rightarrow(2)$. Let $C_{ph}(X)$ be separable space. There is
  a countable dense subset $A=\{f_i\}$ of the space $C_{ph}(X)$.
  Fix $\beta=\{B_j\}$ some countable base for $\mathbb{R}$ consisting of bounded open intervals.

  Consider $\gamma=\{ f^{-1}(\overline{B}): f\in A, B\in \beta\}$.

  We show that a countable family $\gamma$ is required family.


 Let $x_1$ and $x_2$ be distinct points of $X$. Consider open
base set $Q=[\{x_1\},(c_1,d_1)]^{+}\bigcap
[\{x_2\},(c_2,d_2)]^{+}$ of the space $C_{ph}(X)$ where
$(c_i,d_i)\in \beta$ for $i=1,2$ and $\overline{(c_1,d_1)}\bigcap
\overline{(c_2,d_2)}=\emptyset$. There is $h\in Q\bigcap A$.
Clearly that $h^{-1}(\overline{(c_i,d_i)})\in \gamma$, $x_i\in
h^{-1}(\overline{(c_i,d_i)})$ for $i=1,2$ and

$h^{-1}(\overline{(c_1,d_1)})\bigcap
h^{-1}(\overline{(c_2,d_2)})=\emptyset$.

The countable family $\gamma$ is required family.

$(2)\Rightarrow(3)$.  $\gamma=\{Z_i\}$ be the countable family
with required conditions. We can assume that $\gamma$ is closed
under finite unions.

For each $i\in \mathbb{N}$ there is continuous function  $f_{i}: X
\mapsto I=[0,1]$ such that $Z_{i}=f_{i}^{-1}(0)$.

Let $I_{i}=I\times\{i\}$ for every $i\in \mathbb{N}$.

By letting

$(x,i_{1})\textit{E}(y,i_{2})$ wheneve  $x=0=y$ or $x=y$ and
$i_{1}=i_{2}$

we define an equivalence relation  $\textit{E}$ on the set
$\bigcup_{i\in \mathbb{N}} I_{i}$.

The formula

$$\rho([(x,i_{1})],[(y,i_{2})])= \left\{
\begin{array}{rcl}
|x-y|, \,  if  \,  \,   i_{1}=i_{2},\\
x+y, \, if    \, \,   i_{1}\neq i_{2},\\
\end{array}
\right.
$$

defines a metric on the set of equivalence classes of
$\textit{E}$. This space - as well as the corresponding metrizable
space - be called the {\it metrizable hedgehog of spininess
$\aleph_0$} and be denoted $J(\omega)$  (Example 4.1.5 in
~{\cite{eng}}).

 Note that for every  $i\in \mathbb{N}$ the mapping $j_{i}$ of the interval $I$ to $J(\omega)$ defined be letting
  $j_{i}(x)=[(x,i)]$ is a homeomorphic embedding. The family of
  all balls with rational radii around points of the form $[(r,i)]$, where $r$ is a rational number, is a base for
  $J(\omega)$; so that $J(\omega)$ is a separable metrizable
  space.

 The formula $h_{i}(x)=j_{i} (f_{i}(x))$ defines a continuos mapping $h_{i}: X \mapsto J(\omega)$. Note that the family
$\{h_{i}\}_{i=1}^{\infty}$ is functionally separates points of
$X$. Really let $x$ and $y$ be distinct points of $X$. There
exists $Z\in \gamma$ containing exactly one of the points $x$ and
$y$ and there are $i'\in \mathbb{N}$ and continuous function
$f_{i'}:X \mapsto I=[0,1]$ such that $Z=f_{i'}^{-1}(0)$. Hence
$h_{i'}(x)\neq h_{i'}(y)$. Thus diagonal mapping
$h=\triangle_{i\in \mathbb{N}} h_{i} : X \mapsto
J(\omega)^{\omega}$ is a continuous one-to-one mapping from $X$
into the separable metrizable space $J(\omega)^{\omega}$.

It follows that $X$ is a separable submetrizable space.

$(3)\Rightarrow(1)$.  Let $X$ be a separable submetrizable space,
  i.e. $X$ has coarser separable metric topology $\tau_1$ and $\gamma$ be $\pi$-network of $X$ consisting non-trivial
connected sets. Let $\beta=\{B_i\}$ be a countable base of
$(X,\tau_1)$. We can assume that $\beta$ closed under finite union
of its elements.

For each finite family $\{B_{s_i}\}_{i=1}^d\subset \beta$ such
that $\overline{B_{s_i}}\bigcap \overline{B_{s_j}}=\emptyset$ for
$i\neq j$ and $i,j\in \overline{1,d}$ and $\{p_i\}_{i=1}^d \subset
\mathbb{Q}$ we fix $f=f_{s_1,...,s_d,p_1...,p_d}\in C(X)$ such
that $f(\overline{B_{s_{i}}})=p_i$ for each $i=\overline{1,d}$.

Let $G$ be the set of functions $f_{s_1,...,s_d,p_1...,p_d}$ where
$s_i\in\mathbb{N}$ and $p_i\in \mathbb{Q}$ for $i\in\mathbb{N}$.
We claim that the countable set $G$ is dense set of $C_{ph}(X)$.

 By proposition 2.2 in \cite{amk}, let

 $W=[x_1,V_1]^{+}\bigcap...
\bigcap[x_m,V_m]^{+}\bigcap[U_1,r_1]^{-}\bigcap...\bigcap[U_n,r_n]^{-}$
be a base set of $C_{ph}(X)$ where $n,m\in\mathbb{N}$, $x_{i}\in
X$, $V_{i}$ is open set of $\mathbb{R}$ for $i\in \overline{1,m}$,
$U_{j}$ is open set of $X$ and $r_j\in \mathbb{R}$ for $j\in
\overline{1,n}$ and for $i\neq j$, $x_{i}\neq x_{j}$ and
$\overline{U_{i}}\bigcap \overline{U_{j}}=\emptyset$.

Choose $B_{s_l}\in \beta$ for $l=\overline{1,m+n}$ such that

1. $\overline{B_{s_{l_1}}}\bigcap
\overline{B_{s_{l_2}}}=\emptyset$ for $l_1\neq l_2$ and
$l_1,l_2\in \overline{1,n+m}$;

2.  $x_i\in B_{s_l}$ for $l\in \overline{1,m}$;

3.  $B_{s_l}\bigcap U_{k}\neq \emptyset$ for $l\in
\overline{m+1,n+m}$ and $k=l-m$.

 Choose $B_{s'_l}\in \beta$ for
$l\in \overline{1,m}$ such that $x_i\in  B_{s'_l}$ and
$\overline{B_{s'_l}}\subseteq B_{s_l}$.

Choose $A_{k}\in \gamma$ for $k\in \overline{1,n}$ such that
$A_{k}\subseteq (U_k\bigcap B_{s_l})$ where $l=k+m$.

Choose different points $s_{k}, t_{k}\in A_{k}$ for every
$k=\overline{1,m}$.

 Let $S,T\in \beta$ such that
$\overline{S}\bigcap \overline{T}=\emptyset$,
$\overline{B_l}\bigcap \overline{S}=\emptyset$,
$\overline{B_l}\bigcap \overline{T}=\emptyset$ for $l\in
\overline{1,m}$ and $s_{k}\in S$ and $t_{k}\in T$ for all
$k=\overline{1,m}$.

Fix points $v_i\in (V_i\bigcap \mathbb{Q})$ for $i\in
\overline{1,m}$.

Choose $p,q\in \mathbb{Q}$ such that $p<\min\{r_i :
i=\overline{1,n}\}$ and $q>\max\{r_i : i=\overline{1,n}\}$.

Let $$f(x)= \left\{
\begin{array}{rcl}

p  & for &  x\in \overline{S}                              \\
q &  for &  x\in \overline{T}                 \\
v_l & for & x\in \overline{B_{s'_l}} \\
\end{array}
\right.
$$

where $l\in \overline{1,m}$.

Note that $f\in W\bigcap G$. This proves theorem.

\end{proof}

In \cite{os1} the following statements were proved.

\medskip

\begin{theorem}\label{th25}  Let $X$ be a Tychonoff space with countable
$\pi$-base, then the following are equivalent.

\begin{enumerate}

\item $C_{ph}(X)$ is a separable space.

\item $X$ is separable submetrizable space and it has a countable
$\pi$-network consisting of $\mathcal{I}$-sets.

\end{enumerate}

\end{theorem}

\medskip

\begin{theorem}    Let $X$ be a Tychonoff space with countable $\pi$-base, then
the following are equivalent.

\begin{enumerate}

\item $C_{h}(X)$ is a separable space.

\item $X$  has a countable $\pi$-network consisting of
$\mathcal{I}$-sets.

\end{enumerate}
\end{theorem}

\medskip
The next result is corollary of Theorem \ref{th25}, but we notes
its as theorem due to the importance of the class of separable
metrizable spaces.

\begin{theorem}\label{th23} If $X$ is a separable metrizable space, then the following are
equivalent.

\begin{enumerate}

\item $C_{ph}(X)$ is a separable space.

\item $X$ has a countable $\pi$-network consisting of
$\mathcal{I}$-sets.

\end{enumerate}
\end{theorem}

\medskip

We have already noted that if the space $C_{ph}(X)$ is a separable
space then

\medskip

$\bullet$  $X$ is a separable submetrizable;

$\bullet$  $X$ has a $\pi$-network consisting of
$\mathcal{I}$-sets.

\medskip

\begin{theorem}\label{th28} If $X$ is a separable submetrizable space
with countable $\pi$-network consisting of $\mathcal{I}$-sets,
then $C_{ph}(X)$ is separable space.
\end{theorem}

\begin{proof} The proof analogously to the proof of the
implication ($(2)\Rightarrow(1)$) in Theorem 2.4 (\cite{os1}).

 Let $S=\{S_i\}$ be a countable
$\pi$-network of $X$ consisting of  $\mathcal{I}$-sets. By
definition of $\mathcal{I}$-sets, for each $S_{i}\in S$ there is
the continuous function $h_{i}\in C(X)$ such that $h_{i}(S_{i})$
contains an interval $[a_{i},b_{i}]$ of real line. Consider a
countable set

$\{ h_{i,p,q}(x)=\frac{p-q}{a_{i}-b_{i}}*h_{i}(x)+p-
\frac{p-q}{a_{i}-b_{i}}*a_{i}\}$

of continuous functions on $X$, where $i\in \mathbb{N}$, $p,q\in
\mathbb{Q}$. Note that if $h_{i}(x)=a_i$ then $h_{i,p,q}(x)=p$ and
if $h_{i}(x)=b_i$ then $h_{i,p,q}(x)=q$.

 Let $\beta=\{B_j\}$ be countable base of $(X,\tau_1)$
where $\tau_1$ is separable metraizable topology on $X$ because of
$X$ is separable submetrizable space. For each pair
$(B_{j},B_{k})$ such that $\overline{B_{j}}\subseteq B_{k}$ define
continuous functions

$$h_{i,p,q,j,k}(x)= \left\{
\begin{array}{rcl}
h_{i,p,q}(x) \,\, \, \, \, \, for \, \, \,  x\in B_{j}  \\
$\bf{0}$ \, \, \, \, \, \, \, \, \,  for \, \, \,  x\in X\setminus B_{k}.   \\
\end{array}
\right.
$$

 and for each $v\in \mathbb{Q}$

$$d_{j,k,v}(x)= \left\{
\begin{array}{rcl}

v  & for &  x\in B_{j}                              \\
$\bf{0}$ &  for &  x\in X\setminus B_{k}.   \\
\end{array}
\right.
$$

Let $G$ be the set of finite sum of functions $h_{i,p,q,j,k}$ and
$d_{j,k,v}$ where $i,j,k\in\mathbb{N}$ and $p,q,v\in \mathbb{Q}$.
We claim that the countable set $G$ is dense set of $C_{ph}(X)$.

 By proposition 2.2 in \cite{amk}, let

 $W=[x_1,V_1]^{+}\bigcap...\bigcap[x_m,V_m]^{+}\bigcap[U_1,r_1]^{-}\bigcap...\bigcap[U_n,r_n]^{-}$
be a base set of $C_{ph}(X)$ where $n,m\in\mathbb{N}$, $x_{i}\in
X$, $V_{i}$ is open set of $\mathbb{R}$ for $i\in \overline{1,m}$,
$U_{j}$ is open set of $X$ and $r_j\in \mathbb{R}$ for $j\in
\overline{1,n}$ and for $i\neq j$, $x_{i}\neq x_{j}$ and
$\overline{U_{i}}\bigcap \overline{U_{j}}=\emptyset$.

Fix points $y_j\in U_j$ for $j=\overline{1,n}$ and choose
$B_{s_l}\in \beta$ for $l=\overline{1,n+m}$ such that
$\overline{B_{s_{l_1}}}\bigcap \overline{B_{s_{l_2}}}=\emptyset$
for $l_1\neq l_2$ and $l_1,l_2\in \overline{1,n+m}$ and $x_i\in
B_{s_l}$ for $l\in \overline{1,m}$ and $y_j\in B_{s_l}$ for $l\in
\overline{m+1,n}$. Choose $B_{s'_l}\in \beta$ for $l\in
\overline{1,m}$ such that $x_i\in B_{s'_l}$ and
$\overline{B_{s'_l}}\subseteq B_{s_l}$ and choose $B_{s'_l}\in
\beta$ for $l\in \overline{m+1,n+m}$  such that
$y_j\in\overline{B_{s'_l}}\subseteq B_{s_l}$ where $l=j+m$.

Fix points $v_i\in (V_i\bigcap \mathbb{Q})$ for $i\in
\overline{1,m}$ and $p_j,q_j\in \mathbb{Q}$ such that
$p_j<r_j<q_j$ for $j=\overline{1,n}$.

Consider $g\in G$ such that

$g=d_{s'_{1},s_{1},v_1}+...+d_{s'_{m},s_{m},v_m}+h_{i_1,p_1,q_1,s'_{m+1},s_{m+1}}+...+h_{i_n,p_n,q_n,s'_{m+n},s_{n+m}}$
where $S_{i_k}\subset B_{s'_l}\bigcap U_k$ for $k=\overline{1,n}$
and $l=k+m$.

Note that $g\in W\bigcap G$. This proves theorem.

\end{proof}

\begin{corollary} If $X$ is a perfect Polish space, then $C_{ph}(X)$
is separable.
\end{corollary}

\begin{proof} It follows immediately from fact that any regular closed subset of a
space $X$ is a perfect Polish space and it contains some set which
is homeomorphic to $2^{\omega}$ (\cite{kur}). It follows that any
non-empty open set of $X$ is $\mathcal{I}$-set.
\end{proof}

Note that there is the example (Example 4.3. in \cite{os1}) such
that

\medskip

$\bullet$   $X$ hasn't countable chain condition, hence, $X$
hasn't countable $\pi$-network consisting of $\mathcal{I}$-sets;

$\bullet$ $X$ is separable submetrizable space;

$\bullet$ $C_{ph}(X)$ is separable space.

\begin{example} Let $X=\oplus_{\alpha<\mathfrak{c}} \mathbb{R}_{\alpha}$ be
a direct sum of real lines $\mathbb{R}$.
\end{example}

\medskip

In this connection a natural question arises.

\medskip

{\it Question 1.} Let $X$ be a separable submetrizable space with
uncountable $\pi$-network consisting of $\mathcal{I}$-sets. Is
$C_{ph}(X)$ separable ?

\medskip

Recall that a set of reals $X$ is {\it null} if for each positive
$\epsilon$ there exists a cover $\{I_{n}\}_{n\in \mathbb{N}}$ of
$X$ such that $\sum_n diam(I_n)< \epsilon$. A set of reals $X$ has
{\it strong measure zero} if, for each sequence
$\{\epsilon_n\}_{n\in \mathbb{N}}$ of positive reals, there exists
a cover $\{I_{n}\}_{n\in \mathbb{N}}$ of $X$ such that
$diam(I_n)<\epsilon_n$ for all $n$. For example, every Lusin set
has strong measure zero.

\medskip
In \cite{os1} (Example 3.1) was shown that it is consistent with
ZFC
 that
 exists the separable metrizable space $X$ such that  $\Delta(X)=\mathfrak{c}$ and $C_{ph}(X)$
isn't separable.

\medskip

\begin{example}$(\bf CH)$ Let $X$ be a set of reals and it has strong measure zero.
\end{example}

 In \cite{mill} was shown that it is consistent with ZFC that for any set of
 reals of cardinality the continuum, there is a (uniformly) continuous map
 from that set onto the closed unit interval. In fact, this holds
 in the iterated perfect set model.

In \cite{os1} the following statement was proved.

\begin{theorem}( the iterated perfect set model)\label{th20}

If $X$ is a separable metrizable space, then  the following are
equivalent.

\begin{enumerate}

\item $C_{ph}(X)$ is a separable space.

\item  $\Delta(X)=\mathfrak{c}$.

\end{enumerate}

\end{theorem}

In the present paper, we consider more wider form this theorem.

\begin{theorem} ( the iterated perfect set model)\label{th21}
If $X$ is a regular space with a countable network, then the following are equivalent.

\begin{enumerate}

\item $C_{ph}(X)$ is a separable space.

\item  $\Delta(X)=\mathfrak{c}$.

\item $X$ has a countable $\pi$-network consisting of
$\mathcal{I}$-sets.

\end{enumerate}

\end{theorem}

\begin{proof}
$(1)\Rightarrow(2)$. Note that if the space $C_{ph}(X)$ is a
separable space then $C_{p}(X)$ is a separable space and
$C_{h}(X)$ is a separable. It follows that $X$ is a separable
submetrizable space and, hence, $\Delta(X)=\mathfrak{c}$.

$(2)\Rightarrow(3)$. Let $\Delta(X)=\mathfrak{c}$.

(I). We show that any separable metrizable space $M$ of
cardinality $\mathfrak{c}$ is $\mathcal{I}$-set of $M$, i.e. there
exist a continuous function $f:M\mapsto \mathbb{R}$ such that
$f(M)\supseteq \mathcal{I}$.

Really, if a real-valued continuous image of space $M$ has
cardinality less $\mathfrak{c}$ for any $f\in C(M)$, then $M$ is
zero-dimensional space. It follows that $M$ is set of reals and,
by the iterated perfect set model, there is a continuous map from
that set onto the closed unit interval $\mathcal{I}$.

If there is real-valued continuous image of space $M$ such that it
has cardinality $\mathfrak{c}$, then either it contains an
interval $\mathcal{I}$ or it is set of reals and, again, by the
iterated perfect set model, there is a continuous map from that
set onto the closed unit interval $\mathcal{I}$.

(II). Recall that a regular space with a countable network is
normal and separable submetrizable space. Since
$\Delta(X)=\mathfrak{c}$ and $X$ is regular space with a countable
network, it follows that $X$ has countable $\pi$-network $\alpha$
consisting of closed subsets of cardinality $\mathfrak{c}$ of $X$.

We show that $\alpha$ is required $\pi$-network.

Let $f$ be a condensation from $X$ onto a separable metrizable
space. Fix $A\in \alpha$ and consider a mapping
$h=f\upharpoonright A$. By point (I), $h(A)$ is $\mathcal{I}$-set
of $h(A)$, i.e. there exist a continuous function $f:h(A)\mapsto
\mathbb{R}$ such that $f(h(A))\supseteq \mathcal{I}$. Since $X$ is
normal space, by Tietze-Urysohn Extension Theorem, the map $f\circ
h$ can be extended to a real-valued continuous map $F: X \mapsto
\mathbb{R}$. Note that $F(A)=f(h(A))\supseteq \mathcal{I}$ i.e.
$A$ is $\mathcal{I}$-set of $X$.

$(2)\Rightarrow(3)$. It follows from  Theorem \ref{th28}.

\end{proof}

\section{Acknowledgement}

This work was supported by Act 211 Government of the Russian
Federation, contract ¹ 02.A03.21.0006.

\bibliographystyle{model1a-num-names}
\bibliography{<your-bib-database>}







\end{document}